\documentclass[reqno,12pt,a4paper]{amsart}

\usepackage[utf8]{inputenc}
\usepackage[T1]{fontenc}
\usepackage[english]{babel}
\usepackage{lmodern}
\usepackage{geometry}
\usepackage{amsmath,amssymb,amsthm,amsfonts,mathtools}
\usepackage{microtype}
\usepackage{enumitem}
\usepackage{hyperref}
\hypersetup{
  colorlinks=true,
  linkcolor=blue!60!black,
  citecolor=blue!60!black,
  urlcolor=blue!60!black
}
\usepackage{xcolor}

\usepackage{verbatim}
\usepackage{url}
\usepackage{tikz}
\usepackage{setspace}
\usetikzlibrary{arrows}

\newcommand{\Log}{\operatorname{Log}}
\newcommand{\Arg}{\operatorname{Arg}}
\newcommand{\claimA}{\noindent {\bf Claim A.} }
\newcommand{\claimB}{\noindent {\bf Claim B.} }
\newcommand{\proofoftheclaim}{\noindent {\em Proof of the claim.} }

\theoremstyle{plain}
\newtheorem{theorem}{Theorem}[section]
\newtheorem{proposition}[theorem]{Proposition}

\newtheorem{corollary}[theorem]{Corollary}
\newtheorem{definition}[theorem]{Definition}

\newtheorem{example}[theorem]{Example}

\title{Arithmetic properties of arguments of algebraic numbers on the unit circle}

\author[G.~C.~G.~Ferreira]{G.~C.~G.~Ferreira}
\address{Departamento de Matem\'{a}tica\\
Universidade Federal de Ouro Preto\\
Ouro Preto, MG\\
35402--136\\
Brazil\\
}
\email{geraldocesar@ufop.edu.br }

\author[S.~Ribas]{S.~Ribas}
\address{Departamento de Matem\'{a}tica\\
Universidade Federal de Ouro Preto\\
Ouro Preto, MG\\
35402--136\\
Brazil\\
}
\email{savio.ribas@ufop.edu.br }

\date{\today}

\begin{document}

\maketitle

\begin{abstract}
    An irrational number $\theta$ is called Diophantine if there exist $c>0$ and $\tau < \infty$ such that $\left| \theta - \frac{p}{q} \right|  \ge \frac{c}{q^\tau}$ holds for every $(p,q) \in \mathbb{Z} \times \mathbb{N}$. 

    In this paper, we study Diophantine and transcendence properties of some real numbers. Using lower bounds for linear forms in logarithms, we show that if $\beta \in \mathbb{C}$ is an algebraic number with $|\beta|=1$ that is not a root of unity, then $\frac{\Arg(\beta)}{2\pi}$ is Diophantine. We also prove that if $\beta = e^{i\alpha}$ is algebraic, then $\frac{\alpha}{\pi}$ is either rational or transcendental.
    
    As a consequence, we obtain that if $n \ge 2$ is an integer and $\alpha \in \left(0,\frac{\pi}{2}\right)$ satisfies $n \tan \alpha = \tan(n \alpha)$, then $\frac{\alpha}{2\pi}$ is both Diophantine and transcendental, and $\alpha$ is transcendental. This extends a result of \cite{Cyr2012}, which establishes that $\frac{\alpha}{2\pi}$ is irrational.
\end{abstract}

\section{Introduction}

Let $\theta$ be a real number. The {\em irrationality exponent} of $\theta$ is 
$$\mu(\theta) = \sup\left\{ \mu > 0 \colon \left| \theta - \frac{p}{q} \right| < \frac{1}{q^\mu} \text{ holds for infinitely many pairs } (p,q) \in \mathbb{Z} \times \mathbb{N} \right\}.$$
We say that $\theta$ is a {\em Liouville number} whenever $\mu(\theta) = \infty$. On the other hand, $\theta$ is called {\em Diophantine} if there exist $c > 0$ and $\tau < \infty$ such that, for every $(p,q) \in \mathbb{Z} \times \mathbb{N}$, 
\begin{equation}\label{ineq:dioph}
    \left| \theta - \frac{p}{q} \right| \ge \frac{c}{q^\tau}.
\end{equation}
Let $\mathcal L$ be the set of Liouville numbers and let $\mathcal D$ be the set of Diophantine numbers.

We observe that if $\theta = \frac{a}{b}$ is rational, then $\mu(\theta) = 1$, and one may choose $\frac{p}{q} = \frac{a}{b}$ so that \eqref{ineq:dioph} fails to hold. Hence, no rational number is either Diophantine or Liouville. Among irrational numbers, the sets $\mathcal L$ and $\mathcal D$ form a partition.

A very strong result due to Roth \cite{Roth1955} states that $\mu(\theta) = 2$ for every algebraic number $\theta$. In particular, every Liouville number must be transcendental, or equivalently, every algebraic irrational number is Diophantine.

Arithmetic properties such as rationality or irrationality, algebraicity or transcendence,
and Diophantine or Liouville behavior play a fundamental role in several areas of mathematics. They arise naturally in Diophantine approximation, transcendence theory, and number theory, but also have significant implications in dynamics, geometry, and analysis.

In many contexts, good arithmetic properties govern rigidity versus flexibility phenomena and the qualitative behavior of dynamical systems. In what follows, we introduce a simple trigonometric equation whose solutions exhibit
remarkable arithmetic features and give rise to several geometric and dynamical implications,
which will be discussed in the subsequent paragraphs.

Let $n \in \mathbb{Z}$ and consider the equation 
\begin{equation}\label{eq:main}
    n \tan \alpha = \tan(n \alpha).
\end{equation}
If $n\in \{0,1\}$, then every $\alpha \in \mathbb{R}$ is a solution. Furthermore, $\alpha = 0$ is a solution for every $n \in \mathbb{Z}$. Since the tangent function is odd and $\pi$-periodic, we may restrict attention to the case $n \ge 2$ and $\alpha \in \left(0,\frac{\pi}{2}\right)$. In \cite{Cyr2012}, Cyr proved that $\frac{\alpha}{2\pi}$ is irrational.

This equation appears naturally in several topics involving plane curves and billiard dynamics. Despite its elementary appearance, its solutions encode subtle rigidity phenomena with both geometric and dynamical consequences.

One prominent context in which this equation appears is the {\em theory of bicycle curves}. 
In an effort to classify noncircular bicycle curves, Tabachnikov \cite{Tabachnikov2006} proved that the circle admits a nontrivial infinitesimal deformation as a smooth plane bicycle curve with rotation number $\frac{\alpha}{\pi}$ if and only if \eqref{eq:main} holds for some integer $n \ge 2$. Thus, the solutions of this equation characterize the values of the rotation number for which the circle fails to be rigid in this geometric setting. In light of the Cyr theorem, it follows that the circle is rigid as a bicycle curve for any rational rotation number, except in the case $\alpha = \frac{\pi}{2}$, where any curve of constant width is a bicycle curve.

The same equation arises independently in the context of {\em mathematical billiards}. In his study of the dynamics of the billiard map inside the cross section of an infinite cylinder that floats in neutral equilibrium in any orientation, Gutkin \cite{Gutkin2010} showed that for a regular, noncircular billiard table with boundary $\Gamma$, the existence of a constant-angle caustic imposes the arithmetic constraint that there exist an integer $n \ge 2$ and a real parameter $\alpha$ satisfying \eqref{eq:main}. Moreover, under additional assumptions, the arithmetic nature of $\frac{\alpha}{\pi}$ plays a decisive role in determining whether the induced billiard dynamics on the caustic is rational or irrational.

More recently, equations of the form \eqref{eq:main} have also appeared in the study of {\em self-Bäcklund curves}. In this setting, the existence of non-trivial infinitesimal deformations of central conics as a self-Bäcklund is again governed by the solvability of \eqref{eq:main} for certain values of $n \ge 4$ (see \cite[Section~3]{BialyBorTabachnikov2022}).

The examples discussed above show that arithmetic properties of parameters appearing in geometric and dynamical problems can strongly influence their behavior. This perspective motivates the present work, which focuses on the arithmetic properties of real numbers arising in such contexts.

Before stating the main results of this paper, we introduce some additional notation. Let $z \in \mathbb{C} \setminus \{0\}$. We denote by $\Log z = \ln|z| + i\Arg(z)$ the principal branch of the logarithm, where $\Arg(z) \in (-\pi,\pi]$. 

The main results of this paper are the following.

\begin{theorem}\label{thm:general}
    Let $\beta$ be an algebraic number with $|\beta|=1$ which is not a root of unity. Then $\frac{\Arg(\beta)}{2\pi} \in \left(-\frac{1}{2},\frac{1}{2}\right]$ is a Diophantine number.
\end{theorem}

By applying the previous result to the solutions of equation \eqref{eq:main}, together with the classical theorems described in Section \ref{sec:pre}, we obtain the following arithmetic statements.

\begin{theorem}\label{thm:3props}
    Let $n \ge 2$ be an integer and let $\alpha \in \left( 0, \frac{\pi}{2} \right)$ be a solution of \eqref{eq:main}. Then the following statements hold.
    \begin{enumerate}[label=(\alph*)]
        \item $\frac{\alpha}{2\pi}$ is a Diophantine number.
        \item $\frac{\alpha}{2\pi}$ is transcendental.
        \item $\alpha$ is transcendental.
    \end{enumerate}
\end{theorem}

We observe that items $(a)$ and $(b)$ of the previous theorem extend the result of Cyr~\cite{Cyr2012}, while item $(c)$ provides the transcendence of the solution $\alpha$ itself.

The paper is organized as follows. In Section \ref{sec:pre}, we introduce the main tools used in the proof of Theorem \ref{thm:general}, including several algebraic ingredients, as well as the Baker-Wüstholz Theorem, which provides a lower bound for a non-zero linear forms in logarithms, the Gelfond-Schneider Theorem and the Hermite-Lindemann Theorem, which are two remarkable results in transcendence theory. Section \ref{sec:proof} contains the proof of the main result (Theorem \ref{thm:general}). In Section \ref{sec:eqtan}, we apply the main result, together with the remarkable theorems presented in Section \ref{sec:pre}, to the solutions of equation \eqref{eq:main} in order to prove Theorem \ref{thm:3props}. Finally, in Section \ref{sec:conclusion}, we summarize the results obtained in this paper and present an open problem.

\section{Preliminaries}\label{sec:pre}

In this paper, we use the standard notation and properties from algebra without further mention. For instance, if $K, L$ are fields containing $\mathbb{Q}$ and $K \subset L$, then we write $L/K$, and in this case $[L:K]$ denotes the {\em degree} of this extension. We further say that the extension $L/K$ is {\em finite} if $[L:K] < \infty$. If $M$ is an intermediate field, that is, $K \subset M \subset L$, then $[L:K] = [L:M] \cdot [M:K]$ (Tower Property). An element $z \in L$ is {\em algebraic over $K$} if there exists $f(X) \in K[X] \setminus \{0\}$ such that $f(z) = 0$. Otherwise, $z$ is said to be {\em transcendental over $K$}. In the special case that $K = \mathbb{Q}$, we only write {\em algebraic/transcendental}, and denote by $\overline{\mathbb{Q}} \subset \mathbb{C}$ the set of algebraic numbers. The extension $L/K$ is said to be {\em algebraic} if every $z \in L$ is algebraic over $K$. Any finite extension is algebraic, but the converse does not hold in general. If $z$ is algebraic over $K$, then the {\em minimal polynomial} $m_{z,K}(X) \in K[X] \setminus \{0\}$ of $z$ over $K$ is the lowest degree non-zero polynomial that admits $z$ as a zero. We denote by $K(z_1,\dots,z_\ell)$ the smallest extension of $K$ containing $z_1,\dots,z_\ell$. In this case, the Tower Property ensures that $[K(z_1,\dots,z_\ell):K] < \infty$. In particular, if $z$ is algebraic over $K$, then $[K(z):K] = \deg (m_{z,K})$. If $z$ is algebraic, then multiplying $m_{z,\mathbb{Q}}(X)$ by some non-zero integer, we may assume that $m_{z,\mathbb{Q}}(X) \in \mathbb{Z}[X]$. If $m_{z,\mathbb{Q}}(X)$ is monic, then we say that $z$ is an {\em algebraic integer}. For more information on the properties of algebraic numbers, see \cite{DummitFoote,Lang,Milne}.


Let $z \in \mathbb{C}$, $z \neq 0$. Fix the {\em principal branch of the logarithm} $\Log(z) = \ln|z| + i\Arg(z)$, where $\Arg(z) \in (-\pi,\pi]$. If $|z|=1$, then $\Log z = i \Arg(z)$. In particular, $\Log(-1) = i\pi$. The {\em exponential} is defined by $a^b = e^{b \Log(a)}$, where $a,b \in \mathbb{R}$, $a \neq 0$.

\subsection{Baker-Wüstholz Theorem}

In order to prove Theorem \ref{thm:general}, we use the Baker-Wüstholz Theorem, an algebraic result that plays a central role in the theory of linear forms in logarithms. It yields an effective lower bound for expressions like $\Lambda = b_1 \Log(\alpha_1) + \dots + b_m \Log(\alpha_m)$, where $\alpha_1, \dots, \alpha_m$ are algebraic numbers distinct from $0$ and $1$, and $b_1, \dots, b_m \in \mathbb{Z}$. For this, we need the following.

\begin{definition}
    Let $z \in \overline{\mathbb{Q}}$, and let 
    $$m_{z,\mathbb{Q}}(X) = a_0X^d + a_1X^{d-1} + \dots + a_d \in \mathbb{Z}[X]$$
    chosen primitive (that is, $\gcd(a_0,\dots,a_d) = 1$) with $a_0 > 0$. We denote by 
    $$z = z^{(1)}, \; z^{(2)}, \; \dots, \; z^{(d)} \in \mathbb{C}$$
    the conjugates of $z$ (that is, the other zeros of $m_{z,\mathbb{Q}}(X)$). Moreover, for $b_1,\dots,b_m \in \mathbb{Z}$, not all zero, 
    consider the linear form given by 
    $$L(z_1,\dots,z_m) = b_1z_1 + \dots + b_mz_m.$$
    
    \begin{enumerate}[label=\roman*.]
        \item The {\em logarithmic Weil height of $z$} is defined by
        \begin{equation}\label{eq:height}
            h(z) = \frac{1}{d} \left( \ln a_0 + \sum_{j=1}^{d} \ln \max\left\{1,|z^{(j)}|\right\} \right).
        \end{equation}
        In addition, for $z = 0$, we set $h(0) = 0$.

        \item The {\em logarithmic Weil height of $L$} is defined by 
        \begin{equation*}
            h(L) = \ln \max \left\{ |b_1|, \dots, |b_m| \right\}.
        \end{equation*}

        \item The {\em modified logarithmic height of $z$} is defined by 
        \begin{equation*}
            h'(z) = 
            \max\{ h(z), |\Log z|, 1 \}.
        \end{equation*}
    
        \item The {\em modified logarithmic height of $L$} is defined by 
        \begin{equation*}
            h'(L) = 
            \max\{h(L),1\}.
        \end{equation*}
    \end{enumerate}
\end{definition}

It is known \cite[Chapter~3]{Waldschmidt2000} that \eqref{eq:height} does not depend on the primitive representative of $f$. This ensures that the logarithmic Weil height is well defined.

\begin{example}
    \begin{enumerate}[label=\roman*.]
        \item If $z = \frac{p}{q}$ is rational, where $\gcd(p,q)=1$ and $q>0$, then 
        $$h\left(\frac{p}{q}\right) = \ln \max\left\{|p|,q\right\} = h'\left(\frac{p}{q}\right).$$
        \item If $z$ is an algebraic integer, then the minimal polynomial $m_{z,\mathbb{Q}}(X)$ is monic. In this case, 
        $$h(z) = \frac{1}{d} \sum_{j=1}^{d} \ln \max\left\{1,|z^{(j)}|\right\} \quad \text{ and } \quad h'(z) \ge 1.$$ 
        This implies that only conjugates with modulus larger than $1$ contribute to $h(z)$.
        \item If $\zeta \in \mathbb{C}$ is a root of unity, then the conjugates are roots of unity as well, therefore they have modulus $1$. This implies that $h(\zeta) = 0$. In particular, 
        $$h(-1) = 0 \quad \text{ and } \quad h'(-1) = 1.$$ 
        \item For $z \in \mathbb{C}$ and $n \in \mathbb{Z}$, the conjugates of $z^n$ are $(z^{(1)})^n, \dots, (z^{(d)})^n$, counted with multiplicity. Therefore,
        $$h(z^n) = \frac{1}{d} \sum_{j=1}^{d} \ln \max\left\{1,|(z^{(j)})^n|\right\} = |n| \cdot h(z).$$
    \end{enumerate}
\end{example}

We are now in a position to state the main theorem of this section.

\begin{theorem}[Baker-Wüstholz \cite{BakerWuestholz1993}]\label{thm:bw}
    
    Let $\alpha_1, \dots, \alpha_m \in \overline{\mathbb{Q}} \setminus \{0,1\}$, and let $b_1, \dots, b_m \in \mathbb{Z}$, not all zero. Consider the linear form given by 
    $$L(z_1,\dots,z_m) = b_1z_1 + \dots + b_mz_m,$$
    and set 
    $$\Lambda = L(\Log \alpha_1,\dots,\Log \alpha_m) = b_1 \Log(\alpha_1) + \dots + b_m \Log(\alpha_m).$$
    If $\Lambda \neq 0$, then 
    \begin{equation}\label{ineq:Lambda}
        \ln|\Lambda| > -C(m,d) \cdot h'(\alpha_1) \cdot {\dots} \cdot h'(\alpha_m) \cdot h'(L),
    \end{equation}
    where $d = [\mathbb{Q}(\alpha_1,\dots,\alpha_m):\mathbb{Q}]$ and $C(m,d) > 0$ is an effective constant depending only on $m$ and $d$. 
\end{theorem}


The positivity of $C(m,d)$ is essential. Since $h'(\alpha_i) \ge 0$ for every $1 \le i \le m$ and $h'(L) \ge 0$, inequality \eqref{ineq:Lambda} only makes sense if the right-hand side is a negative number (a lower bound for $\ln|\Lambda| < 0$). If $C(m,d)$ were either zero or negative, the statement would become false or trivial. 
More explicit versions of $C(m,d)$ were obtained later, notably by Matveev \cite{Matveev2000}, who gives effective (albeit very large) bounds suitable for concrete applications to Diophantine equations.


\subsection{Gelfond-Schneider Theorem}

A fundamental result in transcendence theory is the Gelfond-Schneider Theorem, which famously resolves Hilbert’s seventh problem. It addresses the arithmetic nature of numbers of the form $a^b$, where $a$ and $b$ are algebraic numbers, and plays a central role in the study of exponential expressions with algebraic parameters, providing a powerful criterion for transcendence.

\begin{theorem}[Gelfond-Schneider 
{\cite[Chapter~X]{Niven1956}}]\label{thm:GS}
    Let $a \in \overline{\mathbb{Q}} \setminus \{0,1\}$ and let $b \in \overline{\mathbb{Q}} \setminus \mathbb{Q}$. Then $a^b$ is transcendental.
\end{theorem}

As a consequence of the previous theorem, we obtain the following result.

\begin{corollary}\label{cor:alpha2pi_dichotomy}
    Let $\alpha \in \mathbb{R}$ and suppose that $\beta = e^{i\alpha} \in \overline{\mathbb{Q}}$. Then $\frac{\alpha}{\pi}$ is either rational or transcendental. In particular, $\frac{\alpha}{\pi}$ cannot be an irrational algebraic number.
\end{corollary}

\begin{proof}
    Let $\alpha \in \mathbb R$ and suppose that $\beta = e^{i\alpha} \in \overline{\mathbb{Q}}$. With the principal branch of the logarithm fixed, it follows that
    $$(-1)^{\frac{\alpha}{\pi}} = e^{\frac{\alpha}{\pi}\Log(-1)} = e^{i\alpha} = \beta.$$
    Suppose that $\frac{\alpha}{\pi} \in \overline{\mathbb{Q}} \setminus \mathbb{Q}$. By Gelfond-Schneider Theorem (Theorem \ref{thm:GS}), $(-1)^{\frac{\alpha}{\pi}}$ is transcendental, a contradiction.
\end{proof}

\subsection{Hermite-Lindemann Theorem}

Another fundamental result in transcendence theory is the Hermite-Lindemann Theorem, which deals with the exponential function evaluated at algebraic arguments. In particular, it implies the transcendence of classical constants such as $e$ and $\pi$. This theorem provides a basic and widely used tool and serves as a foundation for more refined results involving logarithms and linear forms. Here we present a simplified form.

\begin{theorem}[Hermite-Lindemann {\cite[Chapter~IX]{Niven1956}}]\label{thm:HL}
   Let $\alpha \in \overline{\mathbb{Q}} \setminus \{0\}$. Then $e^{\alpha}$ is transcendental. 
\end{theorem}


\section{Proof of Theorem \ref{thm:general}}\label{sec:proof}


Let $\alpha \in (-\pi,\pi]$ and suppose that $\beta = e^{i\alpha} \in \overline{\mathbb{Q}}$ is not a root of unity. Fixed the principal branch of the logarithm, we have that $\Arg(\beta) = \alpha$, hence $\Log(\beta) = i\alpha$.

For $z_1,z_2 \in \mathbb{C}$, $p \in \mathbb{Z}$ and $q \in \mathbb{N}$, set $L(z_1,z_2) = qz_1 - 2pz_2$ and 
$$\Lambda(q,p) = L(\Log(\beta),\Log(-1)) = q \Log(\beta) - 2p \Log(-1)$$
to be a linear form in logarithms with algebraic numbers $\{\beta,-1\}$. 
Furthermore, we have that
\begin{equation}\label{eq:Lqp}
    \Lambda(q,p) = i\left( q\alpha - 2\pi p \right).
\end{equation}
Let $\theta = \frac{\Arg(\beta)}{2\pi} = \frac{\alpha}{2\pi} \in \left( -\frac{1}{2},\frac{1}{2} \right]$. It follows that 
\begin{equation}\label{eq:|Lambda|}
    |\Lambda(q,p)| = |q\alpha - 2\pi p| = 2\pi q \left| \theta - \frac{p}{q} \right|.
\end{equation}

\vspace{1mm}

\claimA $\Lambda(q,p) \neq 0$. 

\vspace{1mm}

\proofoftheclaim Suppose otherwise that $\Lambda(q,p) = 0$. Then $q \Log(\beta) = 2p \Log(-1)$, which implies that $e^{q\Log(\beta)} = e^{2p \Log(-1)}$, and so $\beta^q = 1$, a contradiction. \qed [Claim A.]

\vspace{1mm}

We now apply Baker-Wüstholz Theorem (Theorem \ref{thm:bw}) with $(\alpha_1, \alpha_2, b_1, b_2) = (\beta, -1, q, -2p)$. Let $d = [\mathbb{Q}(\beta):\mathbb{Q}]$. Since 
$$h'(L) = 
\max \left\{ \ln(\max\{q,|2p|\}), 1\right\} = 
\ln(\max\{q,|2p|,e\}) ,$$ 
obtain that 
$$\ln|\Lambda(q,p)|  > -C_0 \ln(\max\{q,|2p|,e\})$$
with $C_0 = C_0(\beta) = C(2,d) h'(\beta) h'(-1) 
> 0$ depending only on $\beta$. Therefore 
$$|\Lambda(q,p)| > (\max\{q,|2p|,e\})^{-C_0}.$$
By \eqref{eq:|Lambda|}, we obtain
\begin{equation}\label{ineq:modulotheta}
    \left| \theta - \frac{p}{q} \right| > \frac{1}{2\pi} \cdot \frac{1}{q(\max\{q,|2p|,e\})^{C_0}} \quad \text{ for every } p \in \mathbb{Z}.
\end{equation}






Notice that $\theta$ is irrational; otherwise $\beta = e^{2\pi i \theta}$ would be a root of unity, which is a contradiction. Therefore, there exists a unique integer $p'$ closest to $q\theta$, that is, $\left|q\theta - p'\right| < \frac{1}{2}$. Since $|p'| < q|\theta| + \frac{1}{2} < \frac{q+1}{2}$, it follows that 
$$\max\{q,2|p'|,e\} < 3q$$
for every $q \in \mathbb{N}$. By \eqref{ineq:modulotheta}, we obtain
$$\left| \theta - \frac{p'}{q} \right| > \frac{1}{2\pi} \cdot \frac{1}{q(3q)^{C_0}} = \frac{1}{2\pi \cdot 3^{C_0}} \cdot \frac{1}{q^{C_0+1}}.$$
Let $c = \frac{1}{2\pi \cdot 3^{C_0}} > 0$ and $\tau = C_0+1 < \infty$. Notice that both $c$ and $\tau$ depend only on $\beta$, and hence only on $\theta$. For every $p \in \mathbb{Z}$ and $q \in \mathbb{N}$, we have that $|q\theta - p| \ge |q\theta - p'|$, therefore
$$\left| \theta - \frac{p}{q} \right| \ge \left| \theta - \frac{p'}{q} \right| \ge \frac{c}{q^\tau}.$$
This is precisely the definition of a Diophantine number.
\qed

\section{The arithmetic aspects of the solutions of $n \tan \alpha = \tan(n \alpha)$}
\label{sec:eqtan}

In this section, we focus on the equation \eqref{eq:main}. As discussed in the introduction, we may assume that $n \ge 2$ and $\alpha \in \left(0,\frac{\pi}{2}\right)$. We begin with some trigonometric tools which will allow us to prove that both $\tan \alpha$ and $e^{i\alpha}$ are algebraic.


Let $t = \tan \alpha$. Using the addition formula $\tan(u+v) = \frac{\tan u + \tan v}{1 - \tan u \tan v}$ and induction on $n$, it follows that there exist polynomials $P_n(t), Q_n(t) \in \mathbb{Z}[t]$ such that 
$$\tan(n \alpha) = \frac{P_n(t)}{Q_n(t)} \quad \text{ for every $n \ge 1$}.$$
Moreover, these polynomials satisfy the recurrence $P_1(t)=t$, $Q_1(t)=1$, 
\begin{equation}\label{eq:PQ}
    P_{n+1}(t) = P_n(t) + tQ_n(t) \quad \text{ and } \quad Q_{n+1}(t) = Q_n(t) - tP_n(t) \quad \text{ for every $n \ge 1$.}
\end{equation}

\begin{proposition}\label{prop:betaalgebraic}
    \begin{enumerate}[label=\roman*.]
        \item If $\alpha$ satisfy \eqref{eq:main}, then $\tan \alpha  \in \overline{\mathbb{Q}}$.
        \item If $t = \tan \alpha  \in \overline{\mathbb{Q}}$, then $\beta = e^{i\alpha} \in \overline{\mathbb{Q}}$.
    \end{enumerate}
\end{proposition}

\begin{proof}
    \begin{enumerate}[label=\roman*.]
        \item From the previous comments, it follows that $P_n(t) - n t Q_n(t) \in \mathbb{Z}[t]$ is a polynomial having $t = \tan \alpha$ as a zero. Moreover, $P_n(t) - n t Q_n(t)$ is non-zero; otherwise \eqref{eq:PQ} would ensure that $\tan(n \alpha) = n \tan \alpha$ for every $n \ge 2$, a contradiction. Therefore $\tan \alpha$ is algebraic.

        \item We have that $\tan \alpha = \frac{e^{i\alpha} - e^{-i\alpha}}{i(e^{i\alpha} + e^{-i\alpha})}$. Thus $\frac{\beta^2-1}{i(\beta^2+1)} = t$ and so $(it-1)\beta^2 + (1+it) = 0$. This implies that $\beta$ satisfies a quadratic equation with coefficients in $\mathbb{Q}(i,t)$. Since $t$ is algebraic, it follows that $\beta$ is algebraic as well.
    \end{enumerate}
\end{proof}

We observe that item i. of the previous proposition relies on the specific form of equation \eqref{eq:main}, whereas item ii. does not. Moreover, the same argument applies to other similar trigonometric equations, as long as one can establish that the corresponding trigonometric parameter is algebraic.

We are now ready to prove Theorem \ref{thm:3props}.

\subsection{Proof of Theorem \ref{thm:3props}(a)}

By Proposition \ref{prop:betaalgebraic}, $\beta = e^{i\alpha} \in \overline{\mathbb{Q}}$ and $|\beta| = 1$. If $\beta$ were a root of unity, say $\beta^q = 1$ for some $q \in \mathbb{N}$, then $1 = e^{i \alpha q} = e^{2\pi i q \frac{\alpha}{2\pi}}$, thus $\frac{\alpha}{2\pi} \in \mathbb{Q}$. This contradicts the result of Cyr \cite{Cyr2012}. Therefore $\beta$ is not a root of unity. Now the proof follows directly from Theorem \ref{thm:general}.
\qed

\subsection{Proof of Theorem \ref{thm:3props}(b)}

Since $\frac{\alpha}{\pi}$ is irrational, this item follows immediately from Corollary \ref{cor:alpha2pi_dichotomy}.
\qed

\subsection{Proof of Theorem \ref{thm:3props}(c)}

If $\alpha \in \overline{\mathbb{Q}}$, then $i\alpha \in \overline{\mathbb{Q}}$. In this case, Hermite-Lindemann Theorem (Theorem \ref{thm:HL}) would imply that $\beta = e^{i\alpha}$ is transcendental, but this contradicts Proposition \ref{prop:betaalgebraic}.
\qed

\section{Concluding remarks}\label{sec:conclusion}

In this paper, we investigated arithmetic properties of real numbers, with particular emphasis on their Diophantine and transcendental nature.

As a consequence of the Baker-Wüstholz Theorem, we proved that if $\beta \in \overline{\mathbb{Q}}$ satisfies $|\beta| = 1$ and is not a root of unity, then $\frac{\Arg(\beta)}{2\pi}$ is Diophantine. Moreover, we showed that if $\beta = e^{i\alpha} \in \overline{\mathbb{Q}}$, then $\frac{\alpha}{\pi}$ is either rational or transcendental.

By combining these results with the classical Gelfond-Schneider and Hermite-Lindemann Theorems, we established strong arithmetic properties for the solutions of the equation
$$n \tan \alpha = \tan(n\alpha),$$
where $n \ge 2$ and $\alpha \in \left(0, \frac{\pi}{2}\right)$. In particular, we proved that the normalized quantity $\frac{\alpha}{2\pi}$ is both Diophantine and transcendental, thereby strengthening the irrationality result of~\cite{Cyr2012}, while $\alpha$ itself is necessarily transcendental. Nevertheless, at present, we have no information on whether $\alpha$ is Diophantine.

\section*{Acknowledgements}

S. Ribas was partially supported by FAPEMIG grants APQ-04712-25 and RED-00133-21, and CNPq grant 420721/2025-8.


\begin{thebibliography}{99}



\bibitem{BakerWuestholz1993}
A.~Baker; G.~Wüstholz,
\emph{Logarithmic forms and group varieties}.
\emph{J. Reine Angew. Math.} \textbf{442} (1993), 19--62.

\bibitem{BialyBorTabachnikov2022}
M.~Bialy, G.~Bor, and S.~Tabachnikov,
\newblock Self-Bäcklund curves in centroaffine geometry and Lamé's equation,
\newblock \emph{Commun. Amer. Math. Soc.} \textbf{2} (2022), 232--282.


\bibitem{Cyr2012}
V.~Cyr,
\newblock A number theoretic question arising in the geometry of plane curves and in billiard dynamics,
\newblock \emph{Proc. Amer. Math. Soc.} \textbf{140} (2012), no.~9, 3035--3040.

\bibitem{DummitFoote}
D.~S.~Dummit and R.~M.~Foote,
\newblock \emph{Abstract Algebra},
\newblock 3rd ed., Wiley, 2004.

\bibitem{Gutkin2010}
E.~Gutkin,
\newblock Capillary floating and the billiard ball problem,
\newblock preprint, arXiv:1012.2448 (2010).

\bibitem{Lang}
S.~Lang,
\newblock \emph{Algebra},
\newblock rev. 3rd ed., Springer, 2002.


\bibitem{Matveev2000}
E.~M.~Matveev,
\emph{An explicit lower bound for a homogeneous rational linear form in the logarithms of algebraic numbers. II}.
\emph{Izvestiya: Mathematics} \textbf{64} (2000), 1217--1269.

\bibitem{Milne}
J.~S.~Milne,
\newblock \emph{Fields and Galois Theory},
\newblock Course notes. Available at: \url{https://www.jmilne.org/math/CourseNotes/FT.pdf}.

\bibitem{Niven1956}
I.~Niven,
\newblock \emph{Irrational Numbers},
\newblock Carus Mathematical Monographs, No.~11, Mathematical Association of America, 1956.

\bibitem{Roth1955}
K.~F.~Roth,
\newblock Rational approximations to algebraic numbers,
\newblock \emph{Mathematika} \textbf{2} (1955), 1--20.





\bibitem{Tabachnikov2006}
S.~Tabachnikov, 
\newblock Tire track geometry: Variations on a theme, 
\newblock \emph{Israel J. Math.} \textbf{151} (2006), 1--28.

\bibitem{Waldschmidt2000}
M.~Waldschmidt,
\newblock \emph{Diophantine Approximation on Linear Algebraic Groups: Transcendence Properties of the Exponential Function in Several Variables},
Grundlehren der mathematischen Wissenschaften, vol.~326,
\newblock Springer-Verlag, Berlin-Heidelberg, 2000.










\end{thebibliography}
\end{document}